\documentclass[12pt]{article}
\usepackage{amsmath,amssymb,amsthm, amsfonts}
\usepackage{etaremune, enumerate, float, verbatim}
\usepackage{algorithm}
\usepackage{algorithmic}
\usepackage{hyperref}
\usepackage{graphicx}
\usepackage{url}
\usepackage[mathlines]{lineno}
\usepackage{dsfont} % needed for double-struck 1
\usepackage{tikz, graphicx, subcaption, caption}% added CR 12/6/17
\usepackage[algo2e,ruled,vlined]{algorithm2e}
\usepackage{mathrsfs,amsmath}
\usepackage{color}
\definecolor{red}{rgb}{1,0,0}

\definecolor{blue}{rgb}{0,0,.7}

\definecolor{green}{rgb}{0,.6,0}

\definecolor{purp}{rgb}{.5,0,.5}

\numberwithin{figure}{section}   % added LH 11/15/17

\newtheorem{thm}{Theorem}[section]
\newtheorem{cor}[thm]{Corollary}
\newtheorem{lem}[thm]{Lemma}

\theoremstyle{definition}

\theoremstyle{definition}

\theoremstyle{definition}

\newcommand{\edim}{\operatorname{edim}}

\newcommand{\bit}{\begin{itemize}}
\newcommand{\eit}{\end{itemize}}
\newcommand{\ben}{\begin{enumerate}}
\newcommand{\een}{\end{enumerate}}
\newcommand{\beq}{\begin{equation}}
\newcommand{\eeq}{\end{equation}}
\newcommand{\bea}{\begin{eqnarray*}} % * means no number
\newcommand{\eea}{\end{eqnarray*}}
\newcommand{\bpf}{\begin{proof}}
\newcommand{\epf}{\end{proof}\ms}
\newcommand{\bmt}{\begin{bmatrix}}
\newcommand{\emt}{\end{bmatrix}}
\newcommand{\ms}{\medskip}

\newcommand{\lc}{\left\lceil}
\newcommand{\rc}{\right\rceil}
\newcommand{\lf}{\left\lfloor}
\newcommand{\rf}{\right\rfloor}

\newcommand{\noi}{\noindent}

\title{Extremal results for graphs of bounded metric dimension}
\author{Jesse Geneson \and Suchir Kaustav \and Antoine Labelle}

\begin{document}
%\linenumbers
\maketitle

\begin{abstract}
Metric dimension is a graph parameter motivated by problems in robot navigation, drug design, and image processing. In this paper, we answer several open extremal problems on metric dimension and pattern avoidance in graphs from (Geneson, Metric dimension and pattern avoidance, Discrete Appl. Math. 284, 2020, 1-7). Specifically, we construct a new family of graphs that allows us to determine the maximum possible degree of a graph of metric dimension at most $k$, the maximum possible degeneracy of a graph of metric dimension at most $k$, the maximum possible chromatic number of a graph of metric dimension at most $k$, and the maximum $n$ for which there exists a graph of metric dimension at most $k$ that contains $K_{n, n}$. 

We also investigate a variant of metric dimension called edge metric dimension and solve another problem from the same paper for $n$ sufficiently large by showing that the edge metric dimension of $P_n^{d}$ is $d$ for $n \geq d^{d-1}$. In addition, we use a probabilistic argument to make progress on another open problem from the same paper by showing that the maximum possible clique number of a graph of edge metric dimension at most $k$ is $2^{\Theta(k)}$. We also make progress on a problem from (N. Zubrilina, On the edge dimension of a graph, Discrete Math. 341, 2018, 2083-2088) by finding a family of new triples $(x, y, n)$ for which there exists a graph of metric dimension $x$, edge metric dimension $y$, and order $n$. In particular, we show that for each integer $k > 0$, there exist graphs $G$ with metric dimension $k$, edge metric dimension $3^k(1-o(1))$, and order $3^k(1+o(1))$. 
\end{abstract}

\noi {\bf Keywords} metric dimension, edge metric dimension, extremal functions, pattern avoidance

\noi{\bf AMS subject classification} 05C12, 05C90 \bigskip

%%%%%%%%%%%%%%%%%%%%%%%%%%%%%%%%%%%%%%%%%%%
\section{Introduction}

The parameter of metric dimension for graphs has been studied for several decades \cite{hmd, lmd, smd, tmd} and is motivated by models of robot navigation \cite{lmd}, drug design \cite{cmd, cmd1, jmd}, and image processing \cite{mmd}. Suppose that a robot is dropped somewhere unknown in a graph and is able to move from vertex to vertex. Some vertices in the graph have landmarks, and the robot can measure its distance in the graph to any landmark. We want to find the fewest number of landmarks so that the robot can determine its current vertex by only using the distances to the landmarks. In other words, we need to find the fewest number of landmarks so that for any two vertices $u, v \in G$, there exists a landmark $\ell$ for which $u$ and $v$ have different distances to $\ell$, i.e. the landmark $\ell$ \emph{distinguishes} $u$ and $v$. We say that a set of landmarks is a \emph{resolving set} for $G$ if any two vertices in $G$ can be distinguished by some landmark in the set. Equivalently, let the \emph{distance vector} of a vertex $v$ be the vector $(d_1,\cdots,d_k)$ where $d_i$ is the distance of $v$ to the $i^{\text{th}}$ landmark. A set of landmarks is a resolving set if and only if no two vertices have the same distance vector. The \emph{metric dimension} $\dim(G)$ of the graph $G$ is the minimum size of a resolving set of $G$.

Recently Kelenc et al \cite{emd} defined a variant of metric dimension in which the robot moves from edge to edge instead of vertex to vertex. In this variant, the robot can still measure its distance in the graph to any landmark, where we define the distance $d(e, w)$ of edge $e = \left\{u, v\right\}$ to vertex $w$ as $\min(d(u, w), d(v, w))$. As with the standard metric dimension, we say that landmark $\ell$ distinguishes edges $e$ and $f$ if $e$ and $f$ have different distances to $\ell$. We say that a set of landmarks is an \emph{edge resolving set} for $G$ if any two edges in $G$ can be distinguished by some landmark in the set. The \emph{edge metric dimension} $\edim(G)$ of the graph $G$ is the minimum size of an edge resolving set of $G$. This variant has already been investigated in several recent publications that focus on generalized Petersen graphs \cite{femd}, necklace graphs \cite{lemd}, sunlet graphs and prism graphs \cite{nemd}, graph operations \cite{yemd}, convex polytopes \cite{zemd}, and extremal values \cite{temd}.

\subsection{Past results and open problems}

Kelenc et al compared $\dim(G)$ and $\edim(G)$ in \cite{emd} and proved that the number of edges in a graph of edge metric dimension $k$ and diameter $D$ is at most $(D+1)^k$. They asked whether there is a bound on $\dim(G)$ in terms of $\edim(G)$ or vice versa. Zubrilina \cite{zmd} showed that $\frac{\edim(G)}{\dim(G)}$ is unbounded and improved the bound on the maximum number of edges in a graph of edge metric dimension $k$ and diameter $D$ to $\binom{k}{2}+k D^{k-1} + D^k$. Zubrilina also characterized the graphs of order $n$ with edge metric dimension $n-1$, asked for a characterization of the graphs of order $n$ with edge metric dimension $n-2$, asked whether there exist graphs $G$ with $\edim(G) \gg 2^{\dim(G)}$, and more generally asked for the triples $(x, y, n)$ for which there exist graphs of metric dimension $x$, edge metric dimension $y$ and order $n$ \cite{zmd}. Geneson \cite{gmd} characterized the graphs of edge metric dimension $n-2$ and improved the bound on the maximum number of edges in a graph of edge metric dimension $k$ and diameter $D$ to $(\lf \frac{2D}{3} \rf + 1)^k + k \sum_{i = 1}^{\lc \frac{D}{3} \rc} (2i)^{k-1}$. Geneson also proved a number of results about metric dimension and pattern avoidance.

Khuller et al \cite{lmd} proved that no graph of metric dimension at most $k$ contains a clique of size $2^k+1$. Kelenc et al \cite{emd} proved that no graph of edge metric dimension at most $k$ contains $K_{1, 2^{k}+1}$. Geneson \cite{gmd} proved that these bounds are sharp, implying that the maximum possible clique number of a graph of metric dimension at most $k$ is $2^k$. Geneson also asked what is the maximum possible clique number of a graph of edge metric dimension at most $k$ \cite{gmd}. This quantity only had rough bounds, a lower bound of $k+1$ and an upper bound of $O(2^{k/2})$.

Geneson \cite{gmd} also proved that the maximum degree of a graph of edge metric dimension at most $k$ is $2^k$. For the same problem with metric dimension, Geneson showed that the maximum degree of a graph of metric dimension at most $k$ is between $3^k-k-1$ and $3^k-1$. As a corollary, this showed that the maximum number of edges in a graph of order $n$ and metric dimension $k$ is at most $\frac{(3^k-1)n}{2}$. 

Furthermore, Geneson showed that the maximum $n$ for which there exists a graph of metric dimension at most $k$ that contains $K_{n,n}$ as a subgraph is in the range $[2^{k/2}-1, 3^k/2]$. Geneson also showed that the maximum possible chromatic number and maximum possible degeneracy of a graph of metric dimension at most $k$ are bounded in $[2^k, 3^k]$ and $[2^{k-1},3^k-1]$ respectively. In addition, Geneson proved that $\edim(P_n^{d}) \leq d$ and asked about the exact value of $\edim(P_n^{d})$ in general.

\subsection{New results}

We introduce a family of infinite graphs $D_k$ for $k = 1, 2, \dots$ in which we can embed all graphs of metric dimension $k$. We use this family to prove a number of extremal results about metric dimension, including that the maximum degree of any graph of metric dimension at most $k$ is $3^k-1$, the maximum number of edges in a graph on $n$ vertices with metric dimension $k$ is $\frac{n(3^k-1)(1-o(1))}{2}$, the maximum $n$ for which there exists a graph of metric dimension at most $k$ that contains $K_{n,n}$ as a subgraph is $n =2^{k-1}$, and the maximum size of a wheel subgraph of a graph of metric dimension at most $k$ is $3^k$. We also use the $D_k$ family to show that the maximum possible chromatic number of any graph of metric dimension at most $k$ is $2^k$ and the maximum possible degeneracy of any graph of metric dimension at most $k$ is $\frac{3^k-1}{2}$. This answers multiple open problems from \cite{gmd}. 

We also answer Zubrilina's question of whether there exist graphs $G$ with $\edim(G) \gg 2^{\dim(G)}$ affirmatively. We learned that independently the paper \cite{imd} answered the same question with a very different construction. Regardless, we use the new construction to make progress on Zubrilina's question of finding the triples $(x, y, n)$ for which there exist graphs of metric dimension $x$, edge metric dimension $y$ and order $n$ by showing that for all $\epsilon > 0$, there exist graphs $G$ for which $\edim(G) \gg 2^{\dim(G)}$ and $\frac{\edim(G)}{|V(G)|} \geq 1-\epsilon$. 

Using a probabilistic argument, we make progress on another problem from \cite{gmd} by proving that the maximum possible clique number of a graph of edge metric dimension at most $k$ is $2^{\Theta(k)}$. We solve another problem from \cite{gmd} for $n$ sufficiently large by proving that $\dim(P_n^{d}) = \edim(P_n^d) = d$ for all $n \geq d^{d-1}$.

\subsection{Structure of the paper}

In Section \ref{dkfam}, we define the family $D_k$ and use it to prove our results on maximum degree, number of edges, chromatic number, degeneracy, complete bipartite subgraphs, and wheels. Section \ref{fromd} has our other pattern avoidance results. In Section \ref{striples}, we provide new constructions of families of graphs with $\edim(G) \gg 2^{\dim(G)}$ and new triples $(x, y, n)$. Finally in Section \ref{sgrid}, we prove that $\dim(P_n^{d}) = \edim(P_n^d) = d$ for all $n \geq d^{d-1}$. In Section \ref{sprob} we conclude and discuss some related open questions.

\section{The family $D_k$} \label{dkfam}

In this section, we define a family of graphs $D_k$ that is very useful for proving extremal results about metric dimension. Let $D_k$ be the graph on the vetrex set $\mathbb{Z}_{\ge 0}^k$ with edges between points that differ by at most one in each coordinate. 

\begin{lem}Any graph of metric dimension $k$ can be embedded as a subgraph of $D_k$ by sending each point to its distance vector with respect to a given resolving set of $k$ landmarks. 
\end{lem}

\begin{proof}
If $G$ has metric dimension $k$, and there is an edge between vertices $u$ and $v$ in $G$, then the distances from $u$ and $v$ to any landmark differ by at most $1$. Thus the images of $u$ and $v$ are connected by an edge in $D_k$, regardless of the set of landmarks that we choose.
\end{proof}

Consider the induced subgraph $C_k(q)$ of $D_k$ whose vertex set consists of the integer points in the $k$-dimensional cross polytope centered at $(q,\cdots,q)$ having as a face the $(k-1)$-simplex with its corners at the $k$ points with all coordinates equal to $q$ except for one coordinate which is equal to $0$.

This graph has metric dimension at most $k$, since if we let $v_i$ be the point with $i^{th}$ coordinate $0$ and all others $q$, then the distance vector of each point $x$ with respect to the set of landmarks consisting of the vertices $v_i$ is exactly $x$.

The following corollary explains the usefulness of $D_k$ for pattern avoidance problems in graphs of bounded metric dimension.

\begin{cor}\label{dkpa}
Given a graph $H$, there exist a graph of metric dimension at most $k$ containing $H$ if and only if $H$ is contained in $D_k$.
\end{cor}

\begin{proof}
The forward direction follows since every graph of metric dimension at most $k$ can be embedded in $D_k$, and the backward direction follows since any copy of $H$ in $D_k$ can be translated to a copy of $H$ in $C_k(q)$ for any sufficiently large $q$.
\end{proof}

\begin{thm}\label{maxdeg}
The maximum possible degree of any graph of metric dimension at most $k$ is $3^k-1$.
\end{thm}

\begin{proof}
It is immediate that the maximum degree of $D_k$ is $3^k-1$, so the maximum possible degree of a graph of metric dimension at most $k$ is $3^k-1$ by Corollary \ref{dkpa} with the family of stars $K_{1, n}$.
\end{proof}

We use $W_n$ to denote the wheel on $n+1$ vertices. In the next result, we show that the maximum size of a wheel subgraph in a graph of metric dimension at most $k$ is the same as the maximum size of a star subgraph. This result also uses the family $D_k$.

\begin{thm}\label{maxwheel}
For $k\ge 2$, the maximum $n$ for which there exists a graph of metric dimension at most $k$ that contains a subgraph isomorphic to the wheel $W_n$ is $n = 3^k-1$.
\end{thm}

\begin{proof}
It suffices to show that the subgraph of $D_k$ on $\left\{0,1,2\right\}^k-\left\{1\right\}^k$ has an Hamiltonian cycle. For $k =2$, we use the cycle $(0,0),(0,1),(0,2),(1,2),(2,2),(2,1),(2,0),(1,0),(0,0)$.

Suppose for inductive hypothesis that the subgraph of $D_k$ on $\left\{0,1,2\right\}^k-\left\{1\right\}^k$ has a Hamiltonian cycle $a_1,a_2,\dots,a_j,a_1$ with $j = 3^k-1$. Then the subgraph of $D_{k+1}$ on $\left\{0,1,2\right\}^{k+1}-\left\{1\right\}^{k+1}$ has the Hamiltonian cycle $0a_1$, $0a_2$, $\dots$, $0a_j$, $0 1^k$, $1a_j$, $1 a_{j-1}$, $\dots$, $1a_2$, $2a_1$, $2a_2$, $\dots$, $2a_j$, $2 1^k$, $1 a_1$, $0a_1$ by the inductive hypothesis.
\end{proof}

We also obtain a sharp bound on the maximum number of edges using the family $C_k(q)$.

\begin{thm}\label{maxedges}
The maximum number of edges in a graph on $n$ vertices with metric dimension $k$ is $\frac{n(3^k-1)(1-o(1))}{2}$.
\end{thm}

\begin{proof}
The upper bound is immediate from Theorem \ref{maxdeg}, while the lower bound follows since the proportion of vertices that are interior in $C_k(q)$ can get arbitrarily close to $1$ as $q \rightarrow \infty$.
\end{proof}

As an immediate corollary, we determine the maximum possible degeneracy of any graph of metric dimension at most $k$.

\begin{thm}
The maximum possible degeneracy of any graph of metric dimension at most $k$ is $\frac{3^k-1}{2}$. 
\end{thm}

\begin{proof}
The upper bound follows since $D_k$ has degeneracy at most $\frac{3^k-1}{2}$: for any subgraph $H$ of $D_k$, the minimal point of $H$ with respect to the lexicographical order has degree at most $\frac{3^k-1}{2}$ in $H$. Thus any graph of metric dimension at most $k$ also has degeneracy at most $\frac{3^k-1}{2}$. The lower bound follows from Theorem \ref{maxedges} for $n$ sufficiently large and the well-known fact that any graph $G$ with $m$ edges and $n$ vertices has degeneracy at least $\frac{m}{n}$.
\end{proof}

$D_k$ is also useful for finding the maximum possible chromatic number. 

\begin{thm}
The maximum possible chromatic number of any graph of metric dimension at most $k$ is $2^k$. 
\end{thm}

\begin{proof}
The lower bound was proved in \cite{gmd}. For the upper bound, note that we can assign a color to each point of $(\mathbb{Z}/2\mathbb{Z})^k$ and color each vertex with the color corresponding to its distance vector modulo $2$. This gives a valid coloring of the graph, since if two vertices are adjacent their distance vectors differ by at most $1$ in each coordinate.
\end{proof}

Next we use the $D_k$ family to solve another open problem from \cite{gmd}, the maximum $n$ for which there exists a graph of metric dimension at most $k$ that contains $K_{n,n}$.

\begin{thm}
The maximum $n$ for which there exists a graph of metric dimension at most $k$ that contains $K_{n,n}$ as a subgraph is $n =2^{k-1}$.
\end{thm}

\begin{proof}
The lower bound is immediate from Corollary \ref{dkpa} by considering the subgraph of $D_k$ on $\left\{0, 1\right\}^k$ with one part having the vertices with first coordinate $0$ and the other part having the vertices with first coordinate $1$.

For the upper bound, suppose we have a copy of $K_{n,n}$ with parts $X$ and $Y$ in a graph of metric dimension at most $k$: we classify the landmarks into two types: $L$ is type A if the set of closest vertices to $L$ in the $K_{n,n}$ only has vertices from $X$, $L$ is type B if the set of closest vertices to $L$ in the $K_{n,n}$ only has vertices from $Y$, and $L$ is type C if the set of closest vertices to $L$ in the $K_{n,n}$ has vertices from both $X$ and $Y$. Then $|X| \leq 2^{|A|+|C|}$ and $|Y| \leq 2^{|B|+|C|}$. If all landmarks are type C, there are only two possible coordinates in the distance vectors for the vertices of the copy of $K_{n,n}$, so there are at most $2^k$ vertices in the $K_{n,n}$, which means $n \leq 2^{k-1}$. On the other hand if some landmark is not type C, then $2^{|A|+|C|} \leq 2^{k-1}$ or $2^{|B|+|C|} \leq 2^{k-1}$.
\end{proof}

We also use $D_k$ to bound the maximum possible minimum degree of a graph of metric dimension at most $k$.

\begin{thm}
The maximum possible minimum degree in any graph of metric dimension at most $k$ is at most $3^{k-1}$. 
\end{thm}

\begin{proof}
Any landmark has degree at most $3^{k-1}$. Indeed, considering the graph as embedded in $D_k$, all neighbours of the $i^\text{th}$ landmark $v_i$ have $i^\text{th}$ coordinate $1$ and differ by at most one from $v_i$ in every other coordinate.
\end{proof}

We can see that the bound in the last theorem is sharp for $k = 2$, using the subgraph of $D_2$ bounded by the square with vertices at $(1, 0), (0, 1), (1, 2), (2, 1)$. It is also sharp for $k = 3$, using a rhombic dodecahedron with the $3$ points $(q,q,0), (q,0,q), (0,q,q)$ at corners of the dodecahedron where four faces meet. This family that maximizes the minimum degree for $k = 2$ and $k = 3$ is similar in structure to the family in \cite{pmd} that maximizes the order of a graph of metric dimension $k$ and diameter $D$. It is unclear, however, whether it could be generalized to higher dimensions while still keeping a minimum degree of $3^{k-1}$.

\section{Further results on metric dimension and pattern avoidance in graphs}\label{fromd}

In \cite{gmd}, Geneson asked what is the maximum possible clique number of a graph of edge metric dimension at most $k$. This quantity was bounded between $k+1$ and $O(2^{k/2})$ in \cite{gmd}. We show next that it is $2^{\Theta(k)}$, using a probabilistic method. 

\begin{thm}
The maximum possible clique number of a graph of edge metric dimension at most $k$ is $2^{\Theta(k)}$.
\end{thm}

\begin{proof}
Consider the set $S$ of ternary strings of length $k$, having digits among $\{0,1,2 \}$. Let $V$ be the subset of $S$ whose digits are among $\{0,1 \}$. Let $+$ denote base $3$ addition. Using the binomial theorem, we can see that there are $\Theta(\sum_{i = 0}^{k}\binom{k}{i}2^{k-i} (2^i)^2) = \Theta(6^k)$ unordered pairs $\left\{p, q\right\}$ with $p = \left\{a, b\right\}$ and $q = \left\{c, d\right\}$ for $a, b, c, d, \in V$ such that $a+b=c+d$. 

If we uniformly at random select a subset $R$ of exactly $t = \Theta((\frac{8}{3})^{\frac{k}{3}})$ elements from $V$, the expected number of $\left\{p, q\right\}$ with $p = \left\{a, b\right\}$ and $q = \left\{c, d\right\}$ such that $a+b=c+d$ and $a, b, c, d \in R$ is $O(t^4\frac{6^k}{16^k})$. We can make this expected number less than $\frac{t}{2}$ using $t = \Theta((\frac{8}{3})^{\frac{k}{3}})$. So there exists a subset $T$ of $V$ having $t = \Theta(( \frac{8}{3})^{\frac{k}{3}})$ distinct elements such that the number of $\left\{p, q\right\}$ with $p = \left\{a, b\right\}$ and $q = \left\{c, d\right\}$ such that $a+b=c+d$ is less than $\frac{t}{2}$. For each of these $\left\{p, q\right\}$ with $p = \left\{a, b\right\}$ and $q = \left\{c, d\right\}$ such that $a+b=c+d$, we remove one of the strings $a, b, c, d$ from $T$, leaving us with a set $U$ of size $\Theta(( \frac{8}{3})^{\frac{k}{3}})$, which has no distinct $\left\{a, b\right\}$ and $\left\{c,d\right\}$ such that $a+b=c+d$.

Now, construct a $K_{|U|}$ and add $2k$ vertices $a_1,a_2,...a_k, b_1, b_2,...,b_k$. Join $a_i$ to a vertex $v$ in the $K_{|U|}$ if the string corresponding to $v$ has $i^{th}$ digit $0$, similarly join $b_i$ to $v$ if the corresponding string has $i^{th}$ digit $1$. We claim that
$a_1,a_2,...a_k,b_1,b_2,...b_k$ uniquely distinguishes the edges of the $K_{|U|}$. First of all, note that the edges of the $K_{|U|}$ are precisely those edges having no coordinate $0$. Because of the way we joined the $a_i, b_i$ to the edges, we can see that it is sufficient to ensure that for any two edges of the $K_{|U|}$, the two sums of the two pairs of strings corresponding to the endpoints of the edges are different. However, $U$ was chosen to satisfy this property. This means that the edge resolving set indeed uniquely distinguishes the edges of the $K_{|U|}$. %Now, for any $1 \leq i \leq k$ we can see that there do not exist four distinct strings $p,q,r,s$ such that $p,q$ differ in their $i^{th}$ digit only and $r,s$ differ in their $i^{th}$ digit only, since this would contradict the property of $U$ as then one of $p+r=q+s, p+s=q+r$ is true. So, by the pigeonhole principle there are at most $k$ pairs of strings in $U$ which differ by exactly one digit. We can delete one vertex from each corresponding pair, this ensures that the remaining edges of the graph are also distinguished by the $a_1,a_2,...a_k, b_1,b_2,...,b_k$.

Thus, we constructed a graph of edge metric dimension at most $2k$ having a $K_n$ with $n = \Theta((\frac{8}{3})^{\frac{k}{3}})$. This shows that the maximum $n$ for which there exists a graph of edge metric dimension at most $k$ containing $K_n$ satisfies $n = \Omega(( \frac{8}{3})^{\frac{k}{6}})$. 
\end{proof}

Up to a constant factor, we also bound the maximum $n$ for which there exists a graph of metric dimension at most $k$ that contains a subgraph isomorphic to the $n$-dimensional hypercube.

\begin{thm}
The maximum $n$ for which there exists a graph of metric dimension at most $k$ that contains a subgraph isomorphic to the $n$-dimensional hypercube $Q_n$ is $n = \Theta(k \log(k))$.
\end{thm}

\begin{proof}
Since the order of a graph of metric dimension $k$ and diameter $D$ is at most $(D+1)^k$, we have $2^n \leq (n+1)^k$, which implies that $n \leq 3k \log_2(k)$. For the lower bound, it is known that $\dim(Q_{\frac{1}{2}k \log_2(k)}) \leq k$ \cite{gridbound}.
\end{proof}

\section{New triples}\label{striples}

In order to prove our result about new triples $(x, y, n)$, we introduce a construction of another family of graphs of metric dimension $k$ that has maximum degree $3^k-1$. Let $L$ be the set of points in $D_k$ with one coordinate $0$ and all others equal to $2$ and let $M_k$ be the induced subgraph of $D_k$ with vertex set $\{1,2,3\}^k\cup L$.

%\begin{defn}
%To construct $M_k$, take one vertex $v_w$ for each word $w$ in $\{1,2,3\}^k$ and $k$ additional vertices $u_i$ for $1\le i \le k$. Put an edge between $v_{22\cdots 2}$ and all others $v_w$ (this is our copy of $K_{1,3^k-1}$). For each $1\le i \le k$, put an edge between the vertex with a $1$ in position $i$ and $2$'s everywhere else and all the vertices with a $2$ at position $i$. Finally, for each $1\le i \le k$, put an edge between $u_i$ and all vertices with a $1$ at position $i$. The $k$ landmarks are the $u_i$: then we can easily see that the distance vector of $v_w$ is exactly $w$.
%\end{defn}

We prove a lemma that we will apply to the family $M_k$ to show that for each integer $k > 0$, there exist graphs $G$ with metric dimension $k$, edge metric dimension $3^k(1-o(1))$, and order $3^k(1+o(1))$.

\begin{lem}\label{highedim}
If $G$ has order $n$ and some vertex of degree $n-1-x$ that is within distance $2$ of all vertices in $G$, then $edim(G) \geq n-1-x-2^x$.
\end{lem}

\begin{proof}
Let $v$ be the vertex of degree $n-1-x$. We may assume that $n-1-x \ge 2^x$, or else the lemma is already trivially true. Let $y_1, \dots, y_{2^x+1}$ denote any $2^x+1$ distinct neighbors of $v$. Let $S = V(G)-\left\{y_1, \dots, y_{2^x+1}\right\}$. To prove the Lemma, it suffices to show that there exist two distinct edges $v y_i$ and $v y_j$ with the same distance vector with respect to $S$, since that would imply that every edge resolving set for $G$ must contain at least $n-1-x-2^x$ neighbors of $v$.

For each $i = 1, \dots, 2^x+1$, $d(v, v y_i) = 0$ and $d(u, v y_i) = 1$ for each neighbor $u$ of $v$ in $S$. Each vertex $u \in S$ that is not a neighbor of $v$ has distance $2$ to $v$, so $d(u, v y_i) \in \left\{1,2\right\}$.

So, for any distinct $y_i$ and $y_j$, the only distances $d(u, v y_i)$ and $d(u, v y_j)$ on which $v y_i$ and $v y_j$ can differ for $u \in S$ are the vertices $u \not \in N[v]$. There are $x$ vertices $u \in S$ that are not in $N[v]$, since $v$ has degree $n-1-x$. So there are $2^x$ distinct possibilities for the tuple of distances from an edge $v y_i$ to $S$, but there are $2^x+1$ vertices $y_i$. Thus by the pigeonhole principle, there exist two distinct edges $v y_i$ and $v y_j$ with the same tuple of distances to $S$, so $S$ is not an edge resolving set for $G$. Thus every edge resolving set for $G$ must contain at least $n-1-x-2^x$ neighbors of $v$, so $\edim(G) \geq n-1-x-2^x$.
\end{proof}

We can use the last lemma on the graphs $M_k$ to obtain the following result which answers the question of Zubrilina of whether there exist graphs $G$ for which $\edim(G) \gg 2^{\dim(G)}$ \cite{zmd}.

\begin{thm}
For all $\epsilon > 0$, there exist graphs $G$ for which $\edim(G) \gg 2^{\dim(G)}$ and $\frac{\edim(G)}{|V(G)|} \geq 1-\epsilon$.
\end{thm}

\begin{proof}
For $k$ sufficiently large, $M_{k}$ is such a graph, since $|V(M_k)| = 3^k+k$, $\edim(M_k) \geq 3^k-1-2^k$ by Lemma \ref{highedim} (with the vertex $(2,2,2)$ and $x=k$) and $\dim(M_k) = k$.
\end{proof}

As a result, we also obtain new triples $(x, y, n)$ of integers for which there exist graphs of metric dimension $x$, edge metric dimension $y$, and order $n$.

\begin{cor}
For each integer $k > 0$, there exist graphs with metric dimension $k$, edge metric dimension $3^k(1-o(1))$, and order $3^k(1+o(1))$.
\end{cor}

Next we provide a construction that shows there is no upper bound on $\edim(G)$ only in terms of $\dim(G)$. This was shown independently by \cite{imd}, but our construction is very different from the one in \cite{imd}, so we include it here.

\begin{thm}
There is no upper bound on $\edim(G)$ only in terms of $\dim(G)$.
\end{thm}

\begin{proof}
Consider the graphs $C_2(q)$ (the induced subgraph of $D_2$ on the square with corners $(0,q)$, $(q,0)$, $(q,2q)$ and $(2q,q)$) for $q \geq 1$. These graphs all have metric dimension $2$, as noted in Section \ref{dkfam}.

Now we show that $\edim(C_2(q)) \geq q-1$. For this, note that, in order to distinguish the two diagonals of a unit square (i.e. pair of edges of the form $\left\{(x,y),(x+1,y+1)\right\}$ and $\left\{(x+1,y),(x,y+1)\right\}$), we need to have a landmark somewhere on the continuation of one of these diagonals. Indeed, the two edges have the same distance to every other point.

In other words, for any integers $q<t,s<3q$ with $t \not\equiv s \pmod{2}$, we must have a landmark on either the line $x+y=t$ or the line $x-y=s$. However, if we have less than $q-1$ landmarks, then there will be some even $q<t<3q$ for which no landmark satisfies $x+y=t$ and, similarly, an odd $q<s<3q$ for which no landmark satisfies $x-y=s$. This is a contradiction, so $\edim(C_2(q))\ge q-1$.

We thus have constructed a family of graphs with metric dimension $2$ and arbitrarily large edge metric dimension, which shows that there cannot be any upper bound on $\edim(G)$ depending uniquely on $\dim(G)$.
\end{proof}

\section{Metric dimension and edge metric dimension of $d$-dimensional grids}\label{sgrid}

We prove in this section that $P_n^d$ has both metric dimension and edge metric dimension $d$ for $n$ sufficiently large. This answers a question from \cite{gmd} for $n$ sufficiently large.

\begin{thm}
If $n \geq d^{d-1}$, then $\dim(P_n^d)=d$.
\end{thm}

\begin{proof}
It is known that for all $d,n$ we have $\dim(P_n^d) \leq d$. For the lower bound, we note that $P_n^d$ has $n^d$ vertices and it's diameter is $(n-1)d$. It is known that graphs with metric dimension $k$ and diameter $D$ have order at most $k+D^k$ \cite{lmd}. Thus $k+((n-1)d)^k \geq n^d$. If $\dim(P_n^d)$ is less than $d$, then it is at most $d-1$, so $(d-1)+((n-1)d)^{d-1} \geq n^d$. But $n \geq d^{d-1}$ implies that $n^d \geq (nd)^{d-1}$. We can see that $(nd)^{d-1}>(d-1)+((n-1)d)^{d-1}$, which means $n^d>(d-1)+((n-1)d)^{d-1}$, which contradicts the earlier inequality $(d-1)+((n-1)d)^{d-1} \geq n^d$. Thus if $n \geq d^{d-1}$, then we have $\dim(P_n^d)=d$.
\end{proof}

Also, we can use the inequality in the last proof to derive the next bound, of independent interest when $n < d^{d-1}$.

\begin{lem}
For general $n,d$ we have $\dim(P_n^d) \geq \frac{d \log(n)}{\log(d(n-1)+1)}$.
\end{lem}

A similar technique also allows us to prove the analogous result for edge metric dimension.

\begin{thm}
If $n \geq d^{d-1}$, then $\edim(P_n^d)=d$.
\end{thm}

\begin{proof}
It is known that $\edim(P_n^d)\le d$ \cite{gmd}. For the lower bound, note that $P_n^d$ has $d (n-1) n^{d-1}$ edges. Graphs with diameter $D$ and edge metric dimension $k$ have at most $(D+1)^k$ edges \cite{emd}. If the edge metric dimension is at most $d-1$, there are at most $(d(n-1)+1)^{d-1}$ edges, so $d (n-1) n^{d-1} \leq (d(n-1)+1)^{d-1}$, but this is clearly false for $n \geq d^{d-1}$ and $d \geq 2$. Thus if $n \geq d^{d-1}$, then $\edim(P_n^d)=d$.
\end{proof}

As before, we can use the inequality in the last proof to derive the next bound.

\begin{lem}
For general $n,d$ we have $\edim(P_n^d) \ge \frac{log(d)+log(n-1)+(d-1)log(n)}{log(d(n-1)+1)}$
\end{lem}

\section{Concluding remarks}\label{sprob}

In this paper, we completely resolved the problems from \cite{gmd} of determining the maximum possible degree of a graph of metric dimension at most $k$, the maximum possible degeneracy of a graph of metric dimension at most $k$, the maximum possible chromatic number of a graph of metric dimension at most $k$, the maximum $n$ for which there exists a graph of metric dimension at most $k$ that contains $K_{n, n}$, and the values of $\dim(P_n^d)$ and $\edim(P_n^d)$ for $n$ sufficiently large with respect to $d$. Using these results, we made new progress on Zubrilina's problem of finding all triples $(x, y, n)$ for which there exists a graph of metric dimension $x$, edge metric dimension $y$, and order $n$. In particular, we proved that for each integer $k > 0$, there exist graphs $G$ with metric dimension $k$, edge metric dimension $3^k(1-o(1))$, and order $3^k(1+o(1))$.

We showed that the maximum possible clique number of a graph of edge metric dimension at most $k$ is $2^{\Theta(k)}$, sharpening the lower bound of $k+1$ from \cite{gmd}. We also showed that the largest wheel in a graph of metric dimension at most $k$ has $3^k$ vertices and the maximum $n$ for which there exists a graph of metric dimension at most $k$ that contains $Q_n$ is $n = \Theta(k \log(k))$. It would be interesting to find the exact value for the maximum possible clique number of a graph of edge metric dimension at most $k$, and and the exact value of the maximum $n$ for which there exists a graph of metric dimension at most $k$ that contains $Q_n$. Another interesting problem is to investigate the maximum $n$ for which there exists a graph of edge metric dimension at most $k$ that contains $Q_n$.

We determined that the maximum possible minimum degree of a graph of metric dimension $k$ is $3^{k-1}$ for $k = 2$ and $k = 3$, but this extremal problem is open for all $k > 3$. For metric dimension and pattern avoidance, we have investigated extremal results for containment and avoidance of complete graphs, complete bipartite graphs, stars, hypercube graphs, and wheels. It would also be interesting to investigate similar extremal results for other families of subgraphs such as balanced spiders, full binary trees, and multidimensional grids. 

\section*{Acknowledgment}
This paper has resulted from the 2020 CrowdMath project on metric dimension (online at \url{http://www.aops.com/polymath/mitprimes2020}). CrowdMath is an open program created jointly by the MIT Program for Research in Math, Engineering, and Science (PRIMES) and the Art of Problem Solving that gives students all over the world the opportunity to collaborate on a research project.

%%%%%%%%

\end{document}